\documentclass[12pt]{amsart}
\usepackage{amsfonts,amssymb,amscd,amstext,mathrsfs,color}

\input xy
\xyoption{all}

\newtheorem{theorem}{Theorem}
\newtheorem{corollary}{Corollary}
\newtheorem{lemma}{Lemma}
\newtheorem{proposition}{Proposition}

\theoremstyle{definition}
\newtheorem{remark}{Remark}

\newcommand{\C}{\mathbb{C}}
\renewcommand{\O}{\mathscr{O}}
\newcommand{\Aut}{{\operatorname{Aut}}}
\newcommand{\End}{{\operatorname{End}}}
\newcommand{\id}{\mathrm{id}}
\renewcommand{\Im}{{\operatorname{Im}}}

\begin{document}

\title{Generic aspects of holomorphic dynamics \\ on highly flexible complex manifolds}

\author{Leandro Arosio and Finnur L\'arusson}

\address{Dipartimento Di Matematica, Universit\`a di Roma \lq\lq Tor Vergata\rq\rq, Via Della Ricerca Scientifica 1, 00133 Roma, Italy}
\email{arosio@mat.uniroma2.it}

\address{School of Mathematical Sciences, University of Adelaide, Adelaide SA 5005, Australia}
\email{finnur.larusson@adelaide.edu.au}

\thanks{L.~Arosio was supported by SIR grant \lq\lq NEWHOLITE -- New methods in holomorphic iteration\rq\rq, no.~RBSI14CFME, and partially supported by the MIUR Excellence Department Project awarded to the Department of Mathematics, University of Rome Tor Vergata, CUP E83C18000100006.  F.~L\'arusson was partially supported by Australian Research Council grant DP150103442.}


\date{11 October 2019}


\begin{abstract}  We prove closing lemmas for automorphisms of a Stein manifold with the density property and for endomorphisms of an Oka-Stein manifold.  In the former case we need to impose a new tameness condition.  It follows that hyperbolic periodic points are dense in the tame non-wandering set of a generic automorphism of a Stein manifold with the density property and in the non-wandering set of a generic endomorphism of an Oka-Stein manifold.  These are the first results about holomorphic dynamics on Oka manifolds.  We strengthen previous results of ours on the existence and genericity of chaotic volume-preserving automorphisms of Stein manifolds with the volume density property.  We build on work of Forn\ae ss and Sibony: our main results generalise theorems of theirs and we use their methods of proof.
\end{abstract}

\maketitle

\section{Introduction} 
\label{sec:intro}

\noindent
This paper continues a line of research begun with our previous paper \cite{AL2019}.  We want to establish dynamical properties that hold for generic endomorphisms or automorphisms of all sufficiently flexible complex manifolds; more precisely, for generic endomorphisms of an Oka-Stein\footnote{By an Oka-Stein manifold we simply mean a complex manifold that is both Oka and Stein.} manifold, for generic automorphisms of a Stein manifold with the density property, or for generic volume-preserving automorphisms of a Stein manifold with the volume density property.

A more specific motivation is to investigate holomorphic dynamics on the manifold $\C^{*n}=(\C\setminus\{0\})^n$, $n\geq 2$.  Dynamics on $\C^{*n}$ is more difficult to understand than the much-studied dynamics on affine space $\C^n$, because $\C^{*n}$ is not known to have the density property and its Haar form $(z_1\cdots z_n)^{-1}dz_1\wedge\cdots\wedge dz_n$ is not exact.  We obtain new results about $\C^{*n}$ as special cases of the general theorems in the paper.  These results are as follows.

\begin{theorem}  \label{t:new-about-C_*^n}
{\rm (a)}  The closing lemma holds for endomorphisms of $\C^{*n}$, $n\geq 1$.

{\rm (b)}  Hyperbolic periodic points are dense in the non-wandering set of a generic endomorphism of $\C^{*n}$, $n\geq 1$.

{\rm (c)}  Chaotic automorphisms are generic among volume-preserving automorphisms of $\C^{*n}$, $n\geq 2$, that act as the identity on $H^{n-1}(\C^{*n})$.  Also, $\C^{*n}$, $n\geq 2$, has a chaotic volume-preserving automorphism that can be approximated by finite compositions of time maps of complete divergence-free vector fields.
\end{theorem}

By a complete vector field we mean a field that can be integrated for all {\it complex} time.

Recall that a point $p$ of a space $X$ is said to be {\it non-wandering} for an endomorphism $f$ of $X$ if for every neighbourhood $U$ of $p$, there is $k\geq 1$ with $U\cap f^k(U)\neq\varnothing$.  Also, $p$ is {\it recurrent} for $f$ if there is a subsequence of $(f^k(p))_{k\in\mathbb N}$ that converges to $p$.  Clearly, a periodic point is recurrent, a recurrent point is non-wandering, and the non-wandering points of $f$ form a closed subset of $X$.

Let $\mathscr E$ be a space of endomorphisms of a space $X$.  We say that the {\it closing lemma} holds for $\mathscr E$ if, whenever $p\in X$ is a non-wandering point of an endomorphism $f$ in $\mathscr E$, every neighbourhood of $f$ in $\mathscr E$ contains an endomorphism of which $p$ is a periodic point.  We say that the {\it weak closing lemma} holds for $\mathscr E$ if, whenever $p\in X$ is a non-wandering point of an endomorphism $f$ in $\mathscr E$, for every neighbourhood $W$ of $f$ in $\mathscr E$ and every neighbourhood $V$ of $p$ in $X$, there is an endomorphism in $W$ with a periodic point in~$V$.

Recall, finally, that an endomorphism is said to be {\it chaotic} if it has a dense forward orbit and its periodic points are dense.  Equivalently, for every pair of nonempty open subsets, there is a cycle that visits both of them.

In their groundbreaking 1997 paper \cite{FS1997}, Forn\ae ss and Sibony established the closing lemma for endomorphisms of $\C^n$ \cite[Theorem 4.1]{FS1997} and automorphisms of $\C^n$ \cite[Theorem 5.1]{FS1997}.  Our first main theorem is a generalisation of their results.  In the automorphism case, we need to impose a new restriction on the non-wandering point.  We express the restriction by saying that the non-wandering point is {\it tame} for the automorphism.  Without this restriction, we do not fully understand the proof of Forn\ae ss and Sibony (our difficulty is with the claim \lq\lq We can even assume that there exists $r>0$ such that \ldots\rq\rq\ at the top of \cite[p.~831]{FS1997}).

Our new tameness property is defined as follows.  For a compact subset $K$ of a complex manifold $X$, we define two closed subsets of $X\times\Aut\, X$:
\[ T_K^+ = \{(x,g)\in X\times\Aut\, X: g^j(x)\in K \textrm{ for all } j\geq 0\}, \]
\[ T_K^- = \{(x,g)\in X\times\Aut\, X: g^j(x)\in K \textrm{ for all } j\leq 0\}. \]
Here, as usual, the group $\Aut\, X$ of holomorphic automorphisms of $X$ is given the compact-open topology.  It is well known that the compact-open topology is defined by a complete metric and makes $\Aut\, X$ a separable topological group.  In summary, $\Aut\, X$ is a Polish group.  

We say that $p\in X$ is {\it tame} for $f\in\Aut\, X$, or that the pair $(p,f)$ is {\it tame}, if whenever $(p,f)\in (T_K^+\cup T_K^-)^\circ$ for a compact subset $K$ of $X$, we have $(p,f)\in \overset\circ {T_L^+} \cup \overset\circ {T_L^-}$ for some, possibly larger, compact subset $L$ of $X$.  Note that $(p,f)$ is tame if and only if $(p,f^{-1})$ is tame.

The following propositions and the subsequent results in this introduction are proved in Section \ref{sec:proofs}.

\begin{proposition}   \label{p:tame-generic}
Let $X$ be a complex manifold.

{\rm (a)}  Tame pairs are generic in $X \times \Aut\, X$.

{\rm (b)}  For a generic automorphism $f$ of $X$, a generic point of $X$ is tame for $f$.
\end{proposition}

\begin{proposition}   \label{p:tame-periodic}
Let $X$ be a Stein manifold and let $p$ be a hyperbolic periodic point of an automorphism $f$ of $X$.  Then the pair $(p,f)$ is tame.
\end{proposition}

We say that the {\it tame closing lemma} holds for automorphisms of a complex manifold $X$ if, whenever $p\in X$ is a non-wandering point of an automorphism $f$ of $X$, {\it and $p$ is tame for} $f$, every neighbourhood of $f$ in $\Aut\, X$ contains an automorphism of which $p$ is a periodic point.

We say that a complex manifold $X$ is {\it homogeneous} if $\Aut\, X$ acts transitively on $X$.  (This is homogeneity in the weaker sense: it is well known that there need not be a transitive Lie group action on a homogeneous manifold.)  Homogeneity is an assumption in one of the closing lemmas below, but what we need in the proof is the related condition that for every neighbourhood $W$ of the identity in $\Aut\, X$ and every point $p\in X$, the set $\{f(p):f\in W\}$ is a neighbourhood of $p$.  This condition, called {\it micro-transitivity} by some authors, implies homogeneity.  It implies that the orbits of the action of $\Aut\, X$ on $X$ are open, so there is only one orbit (we take our manifolds to be connected).  The converse is true by the following theorem of Effros (\cite{Effros1965}; see also \cite{Ancel1987}).  A transitive continuous action of a Polish group on a Polish space is micro-transitive.

The following important sufficient condition for homogeneity is easily established: $X$ is homogeneous if there is a set of complete holomorphic vector fields on $X$ that span each tangent space of $X$.  There is then a finite such set (whether or not $X$ is Stein), so $X$ is elliptic; this was observed by Kaliman and Kutzschebauch \cite[proof of Theorem 4]{KK2008}.  Every Oka-Stein manifold is elliptic and, as far as the present authors are aware, there is no known example of an Oka-Stein manifold on which the complete holomorphic vector fields do not span every tangent space. 

Here, then, is our first main theorem.

\begin{theorem}   \label{t:closing}
{\rm (a)}  The tame closing lemma holds for automorphisms of a Stein manifold with the density property.

{\rm (b)}  The closing lemma holds for endomorphisms of a homogeneous Oka-Stein manifold.

{\rm (c)}  The weak closing lemma holds for endomorphisms of an Oka-Stein manifold.

In each case, the periodic point whose existence is asserted may be taken to be hyperbolic.
\end{theorem}

\begin{remark}
(a)  Let us comment on the 1-dimensional case.  No open Riemann surface has the density property.  The open Riemann surfaces that are Oka are $\C$ and $\C\setminus\{0\}$ and they are homogeneous.  The closing lemma that holds for endomorphisms of each of them seems quite nontrivial, although the proof given here can be simplified somewhat in the 1-dimensional case.

(b)  A Stein manifold with the density property is Oka and homogeneous.  The definitive reference on Oka theory is \cite{Forstneric2017}.  See also the survey \cite{FL2011}.  See \cite[Chapter 4]{Forstneric2017} or \cite{KK2011} for an overview of Anders\'en-Lempert theory, which is the theory of Stein manifolds with the density property or one of its variants.

The prototypical example of a Stein manifold with the density property is $\C^n$, $n\geq 2$.  If $X$ and $Y$ are Stein manifolds with the density property, then so are $X\times Y$, $X\times\C$, and $X\times\C^*$.  Most known examples of Stein manifolds with the density property are captured by the following theorem of Kaliman and Kutzschebauch \cite[Theorem 1.3]{KK2017}.  Let $X$ be a connected affine homogeneous space of a linear algebraic group.  If $X$ is not isomorphic to $\C$ or $\C^{*n}$ for some $n\geq 1$, then $X$ has the algebraic density property and therefore also the density property.

(c)  The manifold $\C^{*n}$, $n\geq 2$, is an important example of an Oka-Stein manifold for which it is not known whether or not it has the density property.  So we know that the closing lemma holds for endomorphisms of $\C^{*n}$, $n\geq 2$, but we do not know whether the tame closing lemma holds for automorphisms.  Automorphisms of $\C^{*n}$, $n\geq 2$, are indeed mysterious: for example, it is an open question whether they necessarily preserve the volume form $(z_1\cdots z_n)^{-1}dz_1\wedge\cdots\wedge dz_n$.  Algebraic automorphisms of $\C^{*n}$ do preserve it; therefore $\C^{*n}$ does not possess the algebraic density property.

(d)  Part (b) of Theorem \ref{t:closing} gives a closing lemma of sorts for automorphisms of a homogeneous Oka-Stein manifold $X$.  If $f$ is an automorphism of $X$ and $p$ is a non-wandering point of $f$, not necessarily tame, then there are endomorphisms of $X$, arbitrarily close to $f$, of which $p$ is a periodic point.  Such endomorphisms will be injective on arbitrarily large compact subsets of $X$.
\end{remark}

Two general density theorems follow rather easily from Theorem \ref{t:closing}.

\begin{corollary}   \label{c:generic-density}
{\rm (a)}  Hyperbolic periodic points are dense in the tame non-wandering set of a generic automorphism of a Stein manifold with the density property.

{\rm (b)}  Hyperbolic periodic points are dense in the non-wandering set of a generic endo\-morphism of an Oka-Stein manifold.
\end{corollary}

As far as the authors know, Theorem \ref{t:closing} and Corollary \ref{c:generic-density} are the first results about holomorphic dynamics on Oka manifolds.

Next we strengthen the main result of our previous paper \cite[Theorem 1]{AL2019}.  There the following result was proved with $A$ equal to the whole group $\Aut_\omega X$ of automorphisms of $X$ that preserve $\omega$, under the hypothesis that $\omega$ is exact.  Clearly, $\Aut_\omega X$ is a closed subgroup of $\Aut\, X$.  The cohomology groups below are complex de Rham cohomology groups.

\begin{theorem}   \label{t:chaotic-generic}
Let $X$ be a Stein manifold of dimension $n\geq 2$ satisfying the volume density property with respect to a holomorphic volume form $\omega$ whose class in $H^n(X)$ lies in the cup product image of $H^1(X)\times H^{n-1}(X)$.  Let $A$ be a closed submonoid of $\Aut_\omega X$ such that:
\begin{itemize}
\item  $A$ contains all time maps of complete divergence-free vector fields on $X$.
\item  Every automorphism in $A$ acts as the identity on $H^{n-1}(X)$.
\end{itemize}
Then chaotic automorphisms are generic in $A$.
\end{theorem}

The largest that $A$ can be is the group of all elements of $\Aut_\omega X$ that act as the identity on $H^{n-1}(X)$ (this group is closed by Poincar\'e duality).  For $X=\C^{*n}$ with the usual Haar form, this largest $A$ is the kernel of an epimorphism $\Aut_\omega X \to \mathrm{SL}(n,\mathbb Z)$.  That the morphism is surjective is shown by the automorphisms
\[ (z_1,\ldots,z_n) \mapsto \big(z_1^{a_{11}}\cdots z_n^{a_{1n}},\ldots, z_1^{a_{n1}}\cdots z_n^{a_{nn}}\big), \]
where $(a_{ij})\in \mathrm{SL}(n,\mathbb Z)$.

The smallest that $A$ can be is the closure of the subgroup of finite compositions of time maps of complete divergence-free vector fields on $X$.  The following corollary is therefore immediate.

\begin{corollary}   \label{c:chaotic-exists}
Let $X$ be a Stein manifold of dimension $n\geq 2$ satisfying the volume density property with respect to a holomorphic volume form whose class in $H^n(X)$ lies in the cup product image of $H^1(X)\times H^{n-1}(X)$.  Then $X$ has a chaotic volume-preserving automorphism that can be approximated by finite compositions of time maps of complete divergence-free vector fields on $X$.  
\end{corollary}

As explained in more detail in \cite{AL2019}, among the Stein manifolds that satisfy the hypotheses of Theorem \ref{t:chaotic-generic} are the following.
\begin{itemize}
\item  Any connected linear algebraic group of dimension at least 2 that is not semisimple, for example $\C^n$ and $\C^{*n}$, $n\geq 2$, with respect to a left- or right-invariant Haar form.  (Theorem \ref{t:chaotic-generic} does not cover any semisimple groups.)
\item  $Y\times\C$ and $Y\times\C^*$, where $Y$ is any Stein manifold with the volume density property.  (A product manifold is always endowed with the product volume form and $\C$ and $\C^*$ carry the standard volume forms $dz$ and $dz/z$, respectively.)
\end{itemize}

Leuenberger \cite{Leuenberger2016} produced new examples of algebraic hypersurfaces in affine space with the density property and the volume density property, including the famous Koras-Russell cubic
\[ C=\{ (x,y,z,w)\in\C^4 : x^2y+x+z^2+w^3=0 \} \]
with the volume form $x^{-2}dx\wedge dz\wedge dw$.  It is known that $C$ is diffeomorphic to $\mathbb R^6$, but not algebraically isomorphic to $\C^3$ (in fact, the algebraic automorphism group does not act transitively on $C$).  Whether $C$ is biholomorphic to $\C^3$ is an open question.  

The Koras-Russell cubic satisfies the tame closing lemma for automorphisms, the closing lemma for endomorphisms, both general density theorems, Theorem \ref{t:chaotic-generic}, and Corollary \ref{c:chaotic-exists}.

\begin{remark}
Recall that the set of periodic points of a chaotic automorphism is dense.  Hence, in the setting of Theorem \ref{t:chaotic-generic}, a closing lemma and a general density theorem can be easily derived.
\end{remark}

Following the proof of Theorem \ref{t:chaotic-generic}, we state a proposition, similar to \cite[Theorem 7]{AL2019}, describing some very particular consequences of the failure of genericity of chaos in the absence of the cohomological hypothesis in Theorem \ref{t:chaotic-generic}, namely the existence of a robustly non-expelling automorphism with a very special orbit.  We end the paper by proving that time maps of global flows are not robustly non-expelling.

\section{Proofs} 
\label{sec:proofs}

\noindent
This section contains the proofs of Theorems \ref{t:closing} and \ref{t:chaotic-generic} and Corollary \ref{c:generic-density}, and, first, Propositions \ref{p:tame-generic} and \ref{p:tame-periodic}.

\begin{proof}[Proof of Proposition \ref{p:tame-generic}]
(a)  Let $K_1\subset K_2\subset\cdots$ be compact sets exhausting $X$, with $K_n\subset\overset\circ K_{n+1}$ for each $n$.  The set of tame pairs in  $X\times\Aut\, X$ contains the set $\Gamma$ of pairs $(p,f)$ such that for each $n$,
\[(p,f)\in (T_{K_n}^+\cup T_{K_n}^-)^\circ \textrm{ implies } (p,f)\in \overset\circ {T_{K_n}^+} \cup \overset\circ {T_{K_n}^-}.\]
We have 
\[\Gamma=\bigcap_n(\overset\circ {T_{K_n}^+} \cup \overset\circ {T_{K_n}^-})\cup ((T_{K_n}^+\cup T_{K_n}^-)^\circ)^\complement,\]
where $\complement$ denotes the complement in $X\times\Aut\, X$, so 
\[\Gamma^\complement=\bigcup_n (\overset\circ {T_{K_n}^+} \cup \overset\circ {T_{K_n}^-})^\complement \cap (T_{K_n}^+\cup T_{K_n}^-)^\circ.\]
For each $n$, the set $(\overset\circ {T_{K_n}^+} \cup \overset\circ {T_{K_n}^-})^\complement \cap (T_{K_n}^+\cup T_{K_n}^-)^\circ$ is contained in $\partial T_{K_n}^+ \cup \partial T_{K_n}^-$.  Since $T_{K_n^+}$ is closed, $\partial T_{K_n}^+$ is closed with empty interior; so is $\partial T_{K_n}^-$.  Hence, by Baire, $\Gamma^\complement$ is contained in an $F_\sigma$ set with empty interior.

(b) is an easy consequence of (a).
\end{proof}

\begin{proof}[Proof of Proposition \ref{p:tame-periodic}]
Let $p$ be a hyperbolic periodic point of an automorphism $f$ of a Stein manifold $X$.  First assume that $p$ is attracting (the case of a repelling periodic point is analogous).  Let $K$ be a compact subset of $X$ containing the $f$-orbit of $p$ in its interior.  The persistence of attracting periodic points implies that for all $(x,g)$ sufficiently close to $(p,f)$, we have $g^j(x)\in K$ for all $j\geq 0$, that is, $(p,f)\in \overset\circ {T_K^+}$.

Now assume that $p$ is a saddle point.  We claim that for every compact subset $K$ of $X$, we have $(p,f)\notin (T_K^+\cup T_K^-)^\circ$, because arbitrarily close to $p$, there are points $x$ with $f^j(x)\notin K$ for some $j<0$ and some $j>0$.  We may assume that $p$ is a fixed point of $f$.  The stable and unstable manifolds $W^s$ and $W^u$ of $p$ are immersed $\C^s$ and $\C^u$, respectively, so neither is contained in $K$.  Take transverse polydiscs $D^u$ to $W^s$ and $D^s$ to $W^u$ outside $K$.  By the lambda lemma \cite[Lemma 2.7.1 and Remark 3, p.~85]{PdM1982}, there are points in $D^u$ that, under iteration by $f$, come arbitrarily close to $p$ and subsequently get mapped into $D^s$.
\end{proof}

Next comes a perturbation lemma that will be used in the proof of Theorem \ref{t:closing}.

\begin{lemma}  \label{l:perturbation}
Let $X$ be a Stein manifold with the density property or with the Oka property, and let $\mathscr S$ be the space of automorphisms or endomorphisms of $X$, respectively.  Let $x_1,\ldots,x_m$ be distinct points in $X$.  For every neighbourhood $W$ of $\id_X$ in $\mathscr S$, there is a neighbourhood $V$ of the identity in the group of linear automorphisms of $T_{x_1} X$, such that for every $\lambda\in V$, there is $h\in W$ such that:
\begin{enumerate}
\item  $h(x_j)=x_j$ for $j=1,\ldots,m$.
\item  $d_{x_j}h=\id$ for $j=2,\ldots,m$.
\item  $d_{x_1}h=\lambda$.
\end{enumerate}
\end{lemma}

\begin{proof}
First consider the endomorphism case.  Embed $X$ as a closed submanifold of $\C^N$ and let $U$ be a tubular neighbourhood of $X$ with a holomorphic retraction $\rho$ onto $X$.  Let $W$ be a neighbourhood of $\id_X$ in $\End\, X$.  (As usual, the monoid $\End\, X$ of holomorphic endomorphisms of $X$ carries the compact-open topology, which is defined by a complete metric.)  Find $\epsilon>0$ and a holomorphically convex compact subset $K$ of $X$ containing $x_1,\ldots,x_m$, such that an endomorphism of $X$ that is within $\epsilon$ of the identity on $K$ lies in $W$ (with respect to, say, the Euclidean metric).  Let $L$ be a compact subset of $X$ containing $K$ in its interior.

It is evident that the lemma holds with $\C^N$ in place of $X$, so given a neighbourhood $W'$ of $\id_{\C^N}$ in $\End\,\C^N$, find a neighbourhood $V'$ of the identity in the group of linear automorphisms of $T_{x_1}\C^N$, such that for every $\lambda\in V'$, there is $f_\lambda\in W'$ satisfying (1--3).  By shrinking $W'$, we may assume that all $f\in W'$ satisfy the following properties.
\begin{itemize}
\item $f(L)\subset U$ and $\rho\circ f$ is within $\epsilon/2$ of the identity on $K$.
\item $\rho\circ f$ and the identity are close enough on $L$ that they are homotopic as maps $L\to X$, so $\rho\circ f$ extends to a continuous map $X\to X$.  (For this, $L$ needs to be well chosen, say as a subcomplex of $X$ with respect to a CW structure on $X$.  Then any continuous map $(L \times [0,1]) \cup (X \times \{0\}) \to X$ extends to a continuous map $X \times [0,1]\to X$.)
\end{itemize}
Consider the neighbourhood $V$ of the identity in the group of linear automorphisms of $T_{x_1} X$ consisting of those automorphisms that extend to an automorphism of $T_{x_1}\C^N$ in $V'$.  Take $\lambda\in V$ and extend it to an automorphism, also denoted $\lambda$, in $V'$.  Let $f_\lambda\in W'$ be as above and let $g_\lambda:X\to X$ be continuous with $g_\lambda=\rho\circ f_\lambda$ on $L$.  Then $g_\lambda$ satisfies (1--3).  Applying the Oka property formulated as the basic Oka property with approximation and jet interpolation, we can deform $g_\lambda$ to a holomorphic map $h_\lambda:X\to X$, still satisfying (1--3), and within $\epsilon/2$ of $g_\lambda$ on $K$.  Thus $h_\lambda$ is within $\epsilon$ of the identity on $K$, so $h_\lambda\in W$.

In the automorphism case, the lemma can be proved in the same way as \cite[Theorem 6]{AL2019}, ignoring preservation of volume.
\end{proof}

Our next lemma isolates the contradiction, due to Forn\ae ss and Sibony, from the end of the proof of \cite[Theorem 5]{AL2019}.  It is used in the proofs of Theorems \ref{t:closing} and \ref{t:chaotic-generic}.  We start with two definitions.

Let $X$ be a complex manifold and let $A$ be a submonoid of $\End\, X$.  We call a holomorphic vector field $\eta$ on $X$ an {\em $A$-velocity} if there is a holomorphic map $\Psi : \C\times X\to X$ such that:
\begin{itemize}
\item $\Psi_t=\Psi(t,\cdot)\in A$ for all $t\in \C$.
\item $\Psi_0= {\rm id}_X$.
\item $\dfrac{\partial}{\partial t}\Psi\bigg\vert_{t=0}=\eta.$
\end{itemize}
We say that $f\in A$ is {\em robustly non-expelling} in $A$ if there is a neighbourhood $W$ of $f$ in $A$, a nonempty open set  $V\subset X$, and a compact subset $K\subset X$ such that $g^j(V)\subset K$ for all $g\in W$ and $j\geq 0$.  Then the closed set 
\[ T_K^+ = \{(x,g)\in X\times A :  g^j(x)\in K \textrm{ for all }  j\geq 0 \} \]
has nonempty interior $U_K$, and the slice $U_{K,f}=\{x\in X : (x,f)\in U_K\}$ is a nonempty, open, relatively compact, forward $f$-invariant subset of $X$.

\begin{lemma}   \label{l:contrad}
Let $X$ be a complex manifold and $A$ be a submonoid of $\End\, X$.  Let $f$ be robustly non-expelling in $A$ and let $H\subset U_{K, f}$ be a nonempty, forward $f$-invariant, compact set. Then there does not exist a continuous, zero-free, $f$-invariant vector field on $H$ which is approximable by $A$-velocities uniformly on $H$.
\end{lemma}

By a vector field on $H$ we mean a section of the tangent bundle of $X$ over $H$.  We say that such a vector field $\xi$ is $f$-invariant if $\xi(f(x))=d_x f(\xi(x))$ for all $x\in H$.

\begin{proof}
We argue by contradiction.  Let $\xi$ be a continuous, zero-free, $f$-invariant vector field on $H$ which is uniformly approximable on $H$ by $A$-velocities.  Let $\lVert\cdot\rVert$ be a hermitian metric on $X$.  Let 
\[ M=\sup\limits_{x\in H, \, j\geq 0}\lVert d_x f^j\rVert, \] 
which is finite since the images of the maps $f^j$ near $H$ are contained in the compact subset $K$ of $X$, and let 
\[ c=\min_{x\in H}\ \lVert \xi(x)\rVert >0. \]
Let $\eta$ be an $A$-velocity such that $\lVert\xi-\eta\rVert\leq \dfrac{c}{2M}$ on $H$.  Let $\Psi : \C\times X\to X$ be associated to $\eta$ as above.  

Since $H\times\{f\}\subset U_K$, there is a neighbourhood $W$ of $f$ in $A$ such that $H\times W\subset U_K$.  Since $\Psi_t\to \id_X$ as $t\to 0$, there is $\delta>0$ such that $\Psi_t\circ f\in W$ when $\lvert t\rvert<\delta$.  Hence 
\[ (\Psi_t\circ f)^j(x)\in K \quad \textrm{when } \lvert t\rvert<\delta, \  x\in H, \  j\geq 0. \]
Cauchy estimates show that there is a constant $C$ such that for all $x\in H$ and $j\geq 0$,
\[ \bigg\lVert\frac{\partial}{\partial t}(\Psi_t\circ f)^j(x)\bigg\vert_{t=0}\bigg\rVert\leq C. \]
By the chain rule, the derivative $\dfrac\partial{\partial t}(\Psi_t\circ f)^j(x)\bigg\vert_{t=0}$ is
\[ \eta(f^j(x))+d_{f^{j-1}(x)}f(\eta (f^{j-1}(x)))+d_{f^{j-2}(x)}f^2(\eta (f^{j-2}(x)))+\cdots. \]
For $x\in H$ and $i=0,\ldots,j-1$,
\[ \lVert d_{f^{j-i}(x)}f^i\big(\eta (f^{j-i}(x))\big)-\xi (f^j(x))\rVert=\lVert d_{f^{j-i}(x)}f^i\big(\eta (f^{j-i}(x))-\xi(f^{j-i}(x))\big)\rVert \leq\frac c 2. \]
Thus we obtain a contradiction.  Namely, for $x\in H$, the derivative $\dfrac\partial{\partial t}(\Psi_t\circ f)^j(x)\bigg\vert_{t=0}$ is bounded as $j\to\infty$, but is also within $\dfrac c 2 j$ of $j\xi(f^j(x))$, whose norm is at least $cj$.
\end{proof}

\begin{proof}[Proof of Theorem \ref{t:closing}]
We largely follow the proofs of Theorems 4.1 and 5.1 in \cite{FS1997}.  Let $X$ be a Stein manifold with the density property or with the Oka property, and let $\mathscr S$ be the space of automorphisms or endomorphisms of $X$, respectively.  Let $p\in X$ be a non-wandering point of $f\in\mathscr S$ that, in the automorphism case, is also tame for $f$.  The proof is divided into two steps.  Before we start, let us dispose of the last claim in Theorem \ref{t:closing} by noting that a periodic point of a morphism can be made hyperbolic by an arbitrarily small perturbation of the morphism using Lemma \ref{l:perturbation}.

\smallskip\noindent
{\bf Step 1.}  Assume that $f$ is robustly non-expelling at $p$, meaning that there is a neighbourhood $W$ of $f$ in $\mathscr S$, a neighbourhood $V$ of $p$ in $X$, and a compact subset $K$ of $X$ such that $g^j(V)\subset K$ for all $g\in W$ and $j\geq 0$.  By replacing $K$ by a larger set, we can assume that it is holomorphically convex.  We will show that $p$ is a periodic point of $f$.

Let $U$ be the interior of the closed set
\[ T=\{(x,g)\in X\times \mathscr S : g^j(x)\in K \textrm{ for all } j\geq 0 \} \]
and let $U_f$ be the slice $\{x\in X:(x,f)\in U\}$ (nonempty by assumption).  Note that each slice $\bigcap\limits_{j\geq 0}g^{-j}(K)$ of $T$ is holomorphically convex.  Clearly, $U_f$ is open and relatively compact, and $f(U_f)\subset U_f$.  We claim that $U_f$ is Runge in $X$.\footnote{We take a Runge open subset to be Stein by definition.}  Let $L\subset U_f$ be compact.  Denote by $\widehat L$ the $\O(X)$-hull of $L$.  Find a compact neighbourhood $L'$ of $L$ and an open neighbourhood $W$ of $f$ such that $L'\times W\subset U$.  Then the neighbourhood $\widehat{L'}\times W$ of $\widehat L\times \{f\}$ is contained in $T$, so $\widehat L\times\{f\}\subset U$ and $\widehat L\subset U_f$.  (To see that $\widehat{L'}$ is indeed a neighbourhood of $\widehat L$, we use the fact that the interior of a holomorphically convex compact subset of a Stein manifold is Runge; see \cite[Proposition 2.7]{FS1992}.)  This shows that $U_f$ is Runge.\footnote{In the same way, the set $A_{f_0}$ in \cite[Claim 2]{AL2019} can be shown to be Runge if we assume, as we may, that the compact sets $\Lambda_k$ are holomorphically convex.  The Runge property is needed in the proof of \cite[Lemma 1]{AL2019}.  The authors neglected to mention this in \cite{AL2019}.}  The connected component $U_0$ of $U_f$ containing $p$ is also Runge.

Since $p$ is non-wandering, there is a smallest integer $\ell\geq 1$ such that $f^\ell(U_0)$ intersects $U_0$, and then $f^\ell(U_0)\subset U_0$.  Then $p$ is non-wandering for $g=f^\ell$.  We claim that $p$ is in fact recurrent for $g$ (and hence for $f$).  To prove this, let $(V_k)$ be a decreasing neighbourhood basis of $p$.  For each $k$, there is $j_k$ such that $g^{j_k}(V_k)$ intersects $V_k$.  Either $(j_k)$ has a strictly increasing subsequence or a constant subsequence (these are not mutually exclusive, of course).  In the latter case, $p$ is a periodic point of $g$ and hence of $f$, so we are done, so let us assume that $(j_k)$ is strictly increasing.  Now $(g^{j_k})$ has a subsequence, say itself, that converges locally uniformly on $U_f$ to a limit $h:U_f\to X$.  If $h(p) \neq p$, then $h(V_k)$ does not intersect $V_k$ for large enough $k$, so $g^{j_k}(V_k)$ does not intersect $V_k$ for large enough $k$, which contradicts the non-wandering assumption.  Hence $h(p)=p$, so $p$ is recurrent for $g$.  

For $j=0,\ldots,\ell-1$, denote by $U_j$ the connected component of $U_f$ containing $f^j(U_0)$.  We have $g(U_j)\subset U_j$.  Let $\Omega=U_0\cup \cdots\cup U_{\ell-1}$.  Clearly, $\Omega$ is Runge and $f(\Omega)\subset \Omega$.  We may assume that the sequence $m_k=j_{k+1}-j_k$ is strictly increasing and that $g^{m_k}$ converges to a holomorphic map $\rho$ on $\Omega$.  Since $g^{m_k}\circ g^{j_k} = g^{j_{k+1}}$, $\rho$ fixes $p$.  Let $M\subset \Omega$ be the subvariety of fixed points of $\rho$, and let $M_0$ be the connected component of $M$ containing $p$.

\noindent
{\it Claim.}  The map $\rho$ satisfies $\rho^2=\rho$ on a neighbourhood of $M_0$.
 Hence $M_0$ is a closed submanifold of $\Omega$.

\noindent{\it Proof.}  Consider the sequence $r_k = m_k-j_k= j_{k+1}-2j_k$, which we may assume is strictly increasing.  A subsequence of $(g^{r_k})$ converges to a map $\eta$ on $U_0$ fixing $p$.  Then
\[ \eta\circ h=\rho \quad\textrm{ on } h^{-1}(U_0),\]
\[ h\circ \rho = h \quad\textrm{ on } \rho^{-1}(U_0).\]
Hence in a neighbourhood of $p$,
\[\rho^2=\eta\circ h\circ \rho=\eta\circ h=\rho,\]
so $\rho^2=\rho$ in a neighbourhood of $M_0$.  Being the image of a holomorphic retraction, $M_0$ is smooth, and the claim is proved.

As $f$ commutes with $\rho$, it follows that $f(M)\subset M$.  Since $g^{j_k}(p)\to p$, we have $g^s(M_0)\subset M_0$ for some $s\geq 1$.  Assume that $\dim M_0\geq 1$; otherwise $p$ is a periodic point for $f$ and we are done.  We will show that this assumption leads to a contradiction.

Choose the smallest  $s$ such that $g^s(M_0)\subset M_0$.  For $j=0,\ldots,s\ell-1$, set $p_j=f^{j}(p)$, and let $M_j$ be the connected component of $M$ containing $p_j$.  Every point $p_j$ satisfies $g^{j_k}(p_j)\to p_j$, so $h(p_j)=p_j$.  Arguing as before, we see that each $M_j$ is a closed submanifold of $\Omega$.  The closed submanifold $\Sigma=M_0\cup\cdots\cup M_{s\ell-1}$ of $\Omega$ (possibly disconnected) is $f$-invariant.  Observe that $m_k$ is a multiple of $s$ for all $k$ large enough.  For each $j$, a sequence of some of the iterates of $g^s$ converges to the identity on $M_j$, so $g^s\vert_{M_j}$ is an automorphism of $M_j$ (it is obviously injective and, by Rouch\'e, surjective), which implies that $f\vert_\Sigma$ is an automorphism of $\Sigma$.

Since $\Sigma$ is a closed submanifold of the relatively compact domain $\Omega$ in $X$, the group $\Aut\, \Sigma$ of holomorphic automorphisms of $\Sigma$ has the structure of a finite-dimensional real Lie group.  Since a sequence of some of the iterates of $f$ converges to the identity on $\Sigma$, the closure $G$ of the subgroup $\{f^n:n\in\mathbb Z\}$, obviously an abelian subgroup of $\Aut\, \Sigma$, is compact (by the lemma of Weil that says that a cyclic subgroup of a locally compact Hausdorff group is either discrete or relatively compact).

Arguing as in the proof of Theorem 5 in \cite{AL2019}, we see that $G$ is isomorphic to a product of a real torus and a finite abelian group, and its orbits are totally real, so the $G$-orbit $Gx$ of any $x\in \Sigma$ is a compact totally real submanifold of $\Sigma$.  Since $\Omega$ is Runge, $\widehat{Gx}\subset \Sigma$.  As in \cite[Claim 6]{AL2019}, it may be shown that there is $q\in\Sigma$ such that $\widehat{Gq}=Gq$.

If $Gq$ is finite, then $q$ is a periodic point of $f$.  Since a sequence of iterates of $f$ converges to the identity on $\Sigma$, $q$ is not attracting.  As in the proof of \cite[Claim 3]{AL2019}, using Lemma \ref{l:perturbation}, we can perturb $f$ so as to obtain an eigenvalue of absolute value strictly bigger than 1, which gives a contradiction.

Assume, finally, that the holomorphically convex totally real submanifold $Gq$ of $\Sigma$ has positive dimension. Let $b\geq 1$ be such that $f^{b}$ is in the identity component $G_0$ of $G$.  There is a $1$-parameter subgroup $(h_t)_{t\in \mathbb{R}}$ of $G_0$ such that $h_1=f^b$.  Consider the vector field $\xi=\dfrac{\partial}{\partial t}h_t\bigg\vert_{t=0}$ on $\Sigma$.  It is holomorphic and it does not have any zeros, for if it did, $f$ would have a periodic point in $\Sigma$.  Morover, since $h_t$ is a limit of iterates of $f$ for each $t$, the vector field $\xi$ is $f$-invariant.  We now obtain a contradiction using Lemma \ref{l:contrad} with $A=\mathscr S$ and $H=Gq$ as soon as we show that $\xi$ is approximable by $\mathscr S$-velocities on $Gq$.

The vector field $\xi$ is holomorphic on $\Sigma$, which is a closed complex submanifold of the Stein open set $\Omega$, so $\xi$ extends holomorphically to $\Omega$.  Since $\Omega$ is a neighbourhood of $Gq$, and $Gq$ is holomorphically convex in $X$, we can approximate $\xi$ uniformly on $Gq$ by a holomorphic vector field $\theta$ on $X$.  

In the automorphism case, since $X$ has the density property, $\theta$ can be approximated uniformly on $Gq$ by a vector field $\eta$ which is the sum of complete holomorphic vector fields $v_1, \dots, v_m$ on $X$.  Let $\varphi^j_t$ be the flow of $v_j$ and let $\Psi_t=\varphi^m_t\circ \cdots\circ \varphi^1_t \in\Aut\, X$, $t\in\C$.  Then $\dfrac{\partial}{\partial t}\Psi_t\bigg\vert_{t=0}=\eta$ on $X$, so $\eta$ is an $\Aut(X)$-velocity approximating $\xi$ on $Gq$. 
  
In the endomorphism case, we let $\Phi$ be the flow of $\theta$, viewed as a holomorphic map from a neighbourhood of $\{0\}\times X$ in $\C\times X$ to $X$.  Since $X$ has the Oka property, $\Phi$ may be approximated on a neighbourhood of $\{0\}\times Gq$ by a holomorphic map $\Psi:\C\times X\to X$ with $\Psi_0=\id_X$.  Then $\dfrac{\partial}{\partial t}\Psi_t\bigg\vert_{t=0}$ is an $\End(X)$-velocity approximating $\xi$ on $Gq$.

In summary, assuming that $f$ is robustly non-expelling at the non-wandering point $p$, we have shown that $p$ is a periodic point of $f$, so the proof is complete in this case.  Note that in the automorphism case, we have not yet used the assumption that $p$ is tame for $f$, and in the endomorphism case, we have not invoked homogeneity of $X$.

\smallskip\noindent
{\bf Step 2.}  Now assume that $f$ is not robustly non-expelling at $p$, meaning that for every neighbourhood $W$ of $f$ in $\mathscr S$, every neighbourhood $V$ of $p$, and every compact subset $K$ of $X$, there is $g\in W$ and a point in $V$ whose $g$-orbit is not contained in $K$.  In the remainder of the proof, the assumption that $p$ is non-wandering is not needed.

First consider the automorphism case.  If $f^{-1}$ is robustly non-expelling at $p$, then we apply Step 1 to $f^{-1}$ in place of $f$, noting that $p$ is a non-wandering point of $f^{-1}$, and conclude that $p$ is a periodic point of $f^{-1}$ and hence of $f$.  So let us assume that neither $f$ nor $f^{-1}$ is robustly non-expelling at $p$.

Let $K\subset X$ be compact and holomorphically convex.  The hypothesis that  neither $f$ nor $f^{-1}$ is robustly non-expelling at $p$ gives $g_1\in \Aut\, X$ and $q_1\in X$ arbitrarily close to $f$ and $p$, respectively, such that the forward $g_1$-orbit of $q_1$ is not contained in $K$, say $g_1^{m_1+1}(q_1)\notin K$, with $m_1\geq 1$ as small as possible, and also gives $g_0\in \Aut\, X$ and $q_0\in X$ arbitrarily close to $f$ and $p$, respectively, such that the backward $g_0$-orbit of $q_0$ is not contained in $K$, say $g_0^{-m_0-1}(q_0)\notin K$, with $m_0\geq 1$ as small as possible.  The assumption that $p$ is tame is designed to allow us to take $q_0=q_1=:q$ and $g_0=g_1=:g$.

We can now use \cite[Theorem 2]{Varolin2000} to find $h\in\Aut\, X$ as close to the identity as we wish on $K$, such that $h$ fixes $g^{-m_0}(q),g^{-m_0+1}(q),\ldots, g^{m_1}(q)$, and $h(g^{m_1+1}(q))=g^{-m_0-1}(q)$.  Then $h\circ g$ has $q$ as a periodic point and $h\circ g$ is as close to $g$ as we wish on $g^{-1}(K)$.  Since $X$ has the density property, it has automorphisms arbitrarily close to the identity that interchange $p$ and any sufficiently nearby point (this can be proved in the same way as \cite[Proposition 1]{AL2019}, ignoring preservation of volume).  To complete the proof of part (a) of the theorem, we conjugate $h\circ g$ by such an automorphism.

Now consider the endomorphism case.  Take neighbourhoods $W$ of $f$ and $V$ of $p$ and a holomorphically convex compact subset $L$ of $X$.  By assumption, there are $g\in W$ and $q\in V$ such that the $g$-orbit of $q$ is not contained in $L$.  Say $g^k(q)\in L$ for $0\leq k<m$ and $g^m(q)\notin L$.  Let $\phi:X\to X$ be continuous, equal to the identity on a neighbourhood of $L$, and with $\phi(g^m(q))=q$.  Since $X$ is Stein and Oka, $\phi$ can be deformed to $h\in\mathscr S$, arbitrarily close to the identity on $L$, such that $h(g^k(q))=g^k(q)$ for $1\leq k<m$ and $h(g^m(q))= q$.  Then $h\circ g$ is arbitrarily close to $g$ on $g^{-1}(L)$ with $q$ as a periodic point.  This concludes the proof of part (c) of the theorem.  

Finally, if $X$ is homogeneous, then $X$ has automorphisms arbitrarily close to the identity that interchange $p$ and any sufficiently nearby point, and we can complete the proof of part (b) of the theorem by conjugating $h\circ g$ by such an automorphism.
\end{proof}

\begin{proof}[Proof of Corollary \ref{c:generic-density}]
We follow the proof of Theorem 6.1 in \cite{FS1997}.  As before, let $X$ be a Stein manifold with the density property or with the Oka property, and let $\mathscr S$ be the space of automorphisms or endomorphisms of $X$, respectively.  

Let $\{U_n:n\geq 1\}$ be a countable basis for the topology of $X$.  Let $S_n$ be the open set of all $f\in\mathscr S$ such that $f$ has a hyperbolic cycle intersecting $U_n$.  Then $G=\bigcap \mathscr S\setminus \partial S_n$ is a residual subset of $\mathscr S$.  We will show that if $f\in G$, then the set $P$ of hyperbolic periodic points of $f$ is dense in the set $\Omega$ of non-wandering points of $f$ that, in the automorphism case, are also tame for $f$.

Suppose that $f\in G$ and $\Omega\cap U_n\neq\varnothing$.  It suffices to show that $P\cap U_n\neq\varnothing$.  By definition of $G$, we have $f\notin\partial S_n$.  Hence either $f\in S_n$, in which case $P\cap U_n\neq\varnothing$ is immediate, or $f\notin\overline S_n$.  The latter case is ruled out by the closing lemma (Theorem \ref{t:closing}).  Indeed, by the closing lemma (for endomorphisms, the weak version suffices), if $p\in \Omega\cap U_n$, then there are morphisms in $\mathscr S$, arbitrarily close to $f$, with hyperbolic periodic points arbitrarily close to $p$.
\end{proof}

Let us recall a definition from \cite{AL2019}.  We call an automorphism {\it expelling} if the set of points with relatively compact forward orbit (this set is $F_\sigma$) has no interior.  Note that a chaotic automorphism is expelling.

\begin{proof}[Proof of Theorem \ref{t:chaotic-generic}]
Let $X$ and $\omega$ be as in Theorem \ref{t:chaotic-generic}.  As in \cite{AL2019}, and closely following the proof of \cite[Theorem 5]{AL2019}, we first prove the weaker result that expelling automorphisms are generic in $A$.  We will argue by contradiction, using Lemma \ref{l:contrad}, so let us assume that expelling automorphisms are not generic in $A$.

Being closed in $\Aut_\omega X$, $A$ is a Baire space.  By Baire category arguments as in the proof of \cite[Claim 1]{AL2019}, there is $f\in A$ which is robustly non-expelling in $A$.  Thus there is a compact set $K\subset X$ such that $U=U_{K,f}$ is a nonempty, open, relatively compact, forward $f$-invariant subset of $X$.  (The set $U_{K,f}$ was defined above, just before the statement of Lemma \ref{l:contrad}.)  By \cite[Claim 2]{AL2019}, $U$ is completely invariant by $f$.

No point $x_0\in U$ can be periodic for $f$.  Indeed, suppose that $x_0$ is periodic for $f$ with period $m$.  Assume that the differential $d_{x_0} f^m$ admits an eigenvalue with absolute value $\lambda$ strictly bigger than $1$.  Then for each $j\geq 1$, the map $f^{mj}: U\to U$ admits an eigenvalue with absolute value $\lambda^j$.  Let $\gamma$ be a holomorphic disc in $U$ tangent to an associated eigenvector.  Then the family of holomorphic discs $f^{mj}\circ \gamma : \mathbb D \to U \Subset X$, $j\geq 1$, contradicts Cauchy estimates at $x_0$.
 
Since $f$ preserves the holomorphic volume form $\omega$, its holomorphic Jacobian determinant is $1$.  Hence, if $d_{x_0} f^m$ has no eigenvalue with absolute value strictly bigger than $1$, then all the eigenvalues of $d_{x_0} f^m$ have absolute value $1$.  Assume this.  Let $v$ be an eigenvector for $d_{x_0} f^m$.  Let $x_j=f^j(x_0)$ for $j\geq 0$.  By \cite[Theorem 6]{AL2019}, there is $h\in \Aut_\omega X$ such that: 
\begin{enumerate}
\item $h$ is arbitrarily close to $\id_X$.
\item $h(x_j)=x_j$ for $j=0,\ldots,m-1$.
\item $d_{x_0}h(v)=\alpha v$, with $\alpha>1$.
\item $d_{x_j}h=\id$ for $j=1,\ldots,m-1$.
\end{enumerate}
From the proof of \cite[Theorem 6]{AL2019}, it is easy to see that the automorphism $h$ is a finite composition of time maps of complete divergence-free holomorphic vector fields, and thus belongs to $A$.  Let $f_1=h\circ f$.  Then $x_0$ is a periodic point of period $m$ for $f_1$ and $v$ is an eigenvector of $d_{x_0}f_1^m$ whose eigenvalue has absolute value strictly greater than $1$.  If $h$ is close enough to $\id_X$, then $f_1$ is close enough to $f$ that the point $(x_0, f_1)$ belongs to $U$.  We obtain a contradiction as before.

Let $\Omega_0$ be a connected component of $U$ and let $\Omega$ be the union of all the connected components in the $f$-cycle of $\Omega_0$.  Clearly $\Omega$ is completely $f$-invariant.  Since $f$ is volume-preserving, $\Omega$ has a finite number of connected components, and since every component is Kobayashi hyperbolic, it follows that $\Aut\,\Omega$ is a real Lie group.  As explained in the proof of Theorem \ref{t:closing}, $\Omega$ can be shown to be Runge.

Let $G$ be the closure in $\Aut\,\Omega$  of the subgroup generated by $f$.  Arguing as in \cite[Claims 4--5--6]{AL2019}, we see that $G$ is compact and abelian, every orbit of $G$ is totally real, and there is an orbit $Gq$ which is  $\O(X)$-convex.  Since $Gq$ is totally real, $\dim Gq\leq n$, and since $f$ admits no periodic points in $U$, $\dim Gq\geq 1$. 
 
The compact set $Gq$ will play the role of $H$ in Lemma \ref{l:contrad}.  The vector field $\xi$ on $Gq$ that will contradict the lemma is provided by the following lemma.

\begin{lemma}
On a neighbourhood of $Gq$, there is a holomorphic vector field $\xi$, which is $f$-invariant, divergence-free, and zero-free, such that the cohomology class of $\xi\rfloor\omega$ in $H^{n-1}(Gq)$ lies in the image of $H^{n-1}(X)$. 
\end{lemma}

\begin{proof}
There are two cases.  Suppose first that $1\leq \dim Gq\leq n-1$.  Take $s\geq 1$ such that $f^s$ lies in the identity component $G_0$ of $G$.  There is a 1-parameter subgroup $(g_t)_{t\in\mathbb R}$ of $G_0$ such that $g_1=f^s$.  Consider the vector field $\xi=\dfrac{d}{dt}g_t\bigg\vert_{t=0}$ on $\Omega$.  It is holomorphic, divergence-free, tangent to the $G$-orbits in $\Omega$, and does not have any zeros, for if it did, $f$ would have a periodic point in $\Omega$, which is absurd.  Moreover, $\xi$ is $f$-invariant, since all elements of $G$ commute with $f$.  The restriction of $\xi\rfloor\omega$ to $Gq$ is the zero form.  This is clear if $\dim Gq\leq n-2$.  If $\dim Gq=n-1$ and we take vectors $v_1,\ldots,v_{n-1}$ in the tangent space of $Gq$ at $x\in Gq$, then the vectors $\xi(x), v_1,\ldots,v_{n-1}$ are linearly dependent because $\xi$ is tangent to $Gq$, so $\xi\rfloor\omega(v_1,\ldots,v_{n-1}) =\omega(\xi(x),v_1,\ldots,v_{n-1})=0$.
 
Suppose now that $\dim Gq=n$.  Then $Gq$ may be identified with a Lie group of the form $\mathbb T^n\times E$, where $E$ is a finite abelian group, embedded as a totally real real-analytic submanifold of $X$.  By the assumption on $\omega$, we may write $\omega = \alpha\wedge\beta+d\gamma$, where $\alpha$ is a closed 1-form on $X$, $\beta$ is a closed $(n-1)$-form, and $\gamma$ is an $(n-1)$-form (these forms may be taken to be holomorphic).  The form $\beta$ cannot be exact on $Gq$.  Indeed, suppose that $\beta$ is exact; then so is $\omega$.  In suitable local holomorphic coordinates $z_1,\ldots,z_n$ on $X$ at each of its points, $Gq$ is defined by the equations $\Im z_1=0,\ldots,\Im z_n=0$, so $\omega\vert_{Gq}$ has no zeros.  On the other hand, since $\omega\vert_{Gq}$ is exact and invariant under the action of $f$ and hence the action of $G$ and hence the action of $Gq$ on itself, we conclude that $\omega\vert_{Gq}=0$, which gives a contradiction. 
 
Since $\beta$ is not exact on $Gq$, its cohomology class $b$ is nonzero in $H^{n-1}(Gq)$.  There is a unique $\mathbb{T}^n$-invariant real-analytic vector field $\xi$ on $Gq$ such that $b=[\xi\rfloor \omega]$.  We claim that $\xi$ is not only $\mathbb{T}^n$-invariant, but also $f$-invariant (or equivalently ($\mathbb{T}^n\times E$)-invariant).  Note first that the vector field ${f}_*\xi$ is $\mathbb{T}^n$-invariant.  Indeed, if $\tau\in\mathbb{T}^n$, then
\[\tau_*({f}_*\xi ) ={f}_*(\tau_*\xi)={f}_*\xi ,\]
since both $\tau$ and $f$ act as elements of the abelian group $G$.  Since $f\in A$ acts as the identity on $H^{n-1}(X)$, restricting to $H^{n-1}(Gq)$, we obtain ${f}_*b=b$, so
\[ [\xi\rfloor \omega]=b={f}_*b=[({f}_*\xi)\rfloor \omega].\]
From the uniqueness of $\xi$ we conclude that ${f}_*\xi=\xi$.  Also, $\xi$ is not identically zero on $Gq$ since $b\neq 0$. From the ($\mathbb{T}^n\times E$)-invariance of $\xi$, it immediately follows that $\xi$ has no zeros on the orbit $Gq$.  Finally, $\xi$ extends to a holomorphic, $f$-invariant vector field on a neighbourhood of $Gq$.  Invariance of $\xi$ and $\omega$ on $Gq$ implies that the flow of $\xi$ preserves $\omega$.  The same holds for the extension of $\xi$, so it is divergence-free.
\end{proof}

In order to prove the genericity of expelling automorphisms in $A$, it remains to show that $\xi$ can be approximated by $A$-velocities on $Gq$.  Let $U_1$ be a tubular neigbourhood of $Gq$ on which $\xi$ is defined.  Then the cohomology class of $\xi\rfloor\omega$ in $H^{n-1}(U_1)$ lies in the image of $H^{n-1}(X)$.  If $U_2\subset U_1$ is a Runge neighbourhood of $Gq$, then the cohomology class of $\xi\rfloor\omega$ in $H^{n-1}(U_2)$ lies in the image of $H^{n-1}(X)$.  Hence we can approximate $\xi$ uniformly on $Gq$ by a divergence-free holomorphic vector field $\tilde\eta$ in $X$.

Since $X$ has the volume density property, $\tilde \eta$ can be approximated uniformly on $H$ by a vector field $\eta$ which is the sum of complete divergence-free holomorphic vector fields $v_1, \dots, v_m$ on $X$.  The vector field $\eta$ is an $A$-velocity.  Indeed, if we let $\varphi^j_t$ be the flow of $v_j$ and set, for $t\in \C$,
\[ \Psi_t=\varphi^m_t\circ \cdots\circ \varphi^1_t \in A, \] 
then
\[ \frac{\partial}{\partial t}\Psi_t(x)\bigg\vert_{t=0}=\eta(x) \quad\textrm{for all }x\in X. \]

We have shown that expelling automorphisms are generic in $A$.  The proof that chaotic automorphisms are generic in $A$ is exactly the same as the proof of Theorem 1 in \cite[Section 5]{AL2019}, noting that the automorphism $g$ that was obtained there from Varolin's \cite[Theorem 2]{Varolin2000} is a composition of time maps and therefore an element of $A$.
\end{proof}

The proof of Theorem \ref{t:chaotic-generic} yields the following result, similar to \cite[Theorem 7]{AL2019}, but now expressing a stronger consequence of the existence of robustly non-expelling volume-preserving automorphisms.

\begin{proposition}
Let $X$ be a Stein manifold of dimension $n\geq 2$ satisfying the volume density property with respect to a holomorphic volume form $\omega$.  Let $A$ be a closed submonoid of $\Aut_\omega X$ such that:
\begin{itemize}
\item  $A$ contains all time maps of complete divergence-free vector fields on $X$.
\item  Every automorphism in $A$ acts as the identity on $H^{n-1}(X)$.
\end{itemize}
Suppose that expelling automorphisms are not generic in $A$.

{\rm (a)}  Then there is an automorphism in $A$ that is robustly non-expelling in $A$.

{\rm (b)}  Every automorphism $f\in A$ that is robustly non-expelling in $A$ has a total orbit whose closure $Z$ is the union of finitely many, mutually disjoint, $n$-dimensional, holomorphically convex, totally real, real-analytic tori in $X$, such that there is no divergence-free, zero-free holomorphic vector field $\xi$ on any neighbourhood of $Z$, such that $\xi$ is $f$-invariant on $Z$ and the cohomology class of $\xi\rfloor\omega$ in $H^{n-1}(Z)$ lies in the image of $H^{n-1}(X)$. 
\end{proposition}

We hope that this result may be of help in constructing examples or proving non-existence of robustly non-expelling automorphisms.  We conclude the paper by showing that examples will not be found among time maps of global flows.

\begin{proposition}
Let $\xi$ be a complete holomorphic vector field without zeros on a Stein manifold $X$.  Let $\Phi_t$, $t\in\C$, be the time maps of the flow of $\xi$.  Then $\Phi_t$ is not robustly non-expelling in $\{\Phi_t:t\in\C\}\subset\Aut\, X$ for any $t\in\C$, that is, $\Phi_t$ is expelling for all $t$ in a dense $G_\delta$ subset of $\C$.
\end{proposition}

\begin{remark}
(a)  As a simple illustration, consider the divergence-free vector field $z\dfrac\partial{\partial z}+w\dfrac\partial{\partial w}$ on $\C^{*2}$ with $\Phi_t(z,w)=(e^t z, e^t w)$.  If $t$ is imaginary, $\Phi_t$ is not expelling (every orbit is relatively compact), but if $t$ is not imaginary, $\Phi_t$ is expelling.

(b)  If $X$ has the volume density property and $\xi$ is divergence-free, then, without the assumption that $\xi$ has no zeros, we can prove that $\Phi_t$ is not robustly non-expelling in the closure of the subgroup of finite compositions of time maps of complete divergence-free vector fields on $X$, because fixed points of a robustly non-expelling volume-preserving automorphism can be ruled out as in the proof of Theorem \ref{t:chaotic-generic}.
\end{remark}

\begin{proof}
Suppose that $\Phi_1$, say, is robustly non-expelling in $\{\Phi_t:t\in\C\}$, so there are $\delta>0$, a nonempty open subset $V$ of $X$, and a compact subset $K$ of $X$ such that $\Phi_{t+1}^j(V)\subset K$ for all $j\geq 0$ and $\lvert t\rvert<\delta$.  Take $p\in V$.  As in the proof of Lemma \ref{l:contrad}, and with the same notation as there, on the one hand, there is a constant $C$ such that
\[ \bigg\lVert\frac{\partial}{\partial t}(\Phi_t\circ \Phi_1)^j(p)\bigg\vert_{t=0}\bigg\rVert\leq C \quad \textrm{for all }j\geq 0. \]
On the other hand, 
\[ \dfrac{\partial}{\partial t}(\Phi_t\circ \Phi_1)^j(p)\bigg\vert_{t=0} = j\xi(\Phi_1^j(p))\] 
and 
\[ \lVert \xi(\Phi_1^j(p)) \rVert \geq \inf\limits_K\lVert\xi\rVert>0. \qedhere \]
\end{proof}

\end{document}